\documentclass[runningheads]{llncs}
\usepackage{graphicx}
\usepackage{amsmath}
\usepackage{amssymb}
\usepackage{xcolor}
\usepackage{url}
\newif\ifdraft
\draftfalse
\newcommand{\mohit}[1]{\ifdraft{\color{cyan}[{\bf Mohit}: #1]}\fi}

\newcommand{\jb}[1]{\ifdraft{\color{blue}[{\bf JB}: #1]}\fi}
\usepackage{amsmath}

\let\llncssubparagraph\subparagraph
\let\subparagraph\paragraph
\usepackage[compact]{titlesec}
\let\subparagraph\llncssubparagraph

\usepackage{titlesec}

\titlespacing*{\section}{0pt}{0.6ex}{0.6ex}
\titlespacing*{\subsection}{0pt}{0.6ex}{0.6ex}
\titlespacing*{\subsubsection}{0pt}{0.4ex}{0.4ex}

\begin{document}
\title{A formal proof of the Lax equivalence theorem for finite difference schemes}
\author{Mohit Tekriwal \and Karthik Duraisamy \and
Jean-Baptiste Jeannin}
\authorrunning{M. Tekriwal et al.}
\institute{University of Michigan, Ann Arbor, MI 48109, USA \\
\email{\{tmohit,kdur,jeannin\}@umich.edu}}
\maketitle              % typeset the header of the contribution

\begin{abstract}
The behavior of physical systems is typically modeled using differential equations which are too complex to solve analytically. In practical problems, these equations are discretized on a computational domain, and numerical solutions are computed.  A numerical scheme is called convergent, if in the limit of infinitesimal discretization, the bounds on the discretization error is also infinitesimally small. The approximate solution converges to the ``true solution'' in this limit. The Lax equivalence theorem enables a proof of convergence given consistency and stability of the method.

In this work, we formally prove the Lax equivalence theorem using the Coq Proof Assistant. We assume a continuous linear differential operator between complete normed spaces, and define an equivalent mapping in the discretized space. Given that the numerical method is consistent (i.e., the discretization error tends to zero as the discretization step tends to zero), and the method is stable (i.e., the error is uniformly bounded), we formally prove that the approximate solution converges to the true solution. We then demonstrate convergence of the difference scheme on an example problem by proving both its consistency and stability, and then applying the Lax equivalence theorem.  In order to prove consistency, we use the Taylor--Lagrange theorem by formally showing that the discretization error is bounded above by the $n^{th}$ power of the discretization step, where $n$ is the order of the truncated Taylor polynomial. 

\end{abstract}
\section{Introduction}
%\jb{Page limit 15 pages excluding bibliography and clearly marked appendices}

Physical systems are typically modeled by differential equations. For instance, the aerodynamics of an airplane can be represented by the Navier--Stokes equations~\cite{NavierSt76:online}, which are too complex to solve analytically.

Since analytical solutions are intractable for most practical problems of interest, numerical solutions are sought  in a discretized domain. The process of discretization in space and time results in approximate solutions to the governing equations.
A numerical scheme is called \textit{convergent}, if in the limit of infinitesimal discretization, the bound on the discretization error is also infinitesimally small. Under these conditions, the numerical solution converges or approaches the analytic solution. This idea is formally articulated by the Lax equivalence theorem~\cite{lax1956survey}, which states that if a numerical method is \textit{consistent} and \textit{stable}, then it is \textit{convergent}.  
%\jb{These are all standard definitions, we also need to focus on formal verification and explain why it's useful, and why we're doing this.}
%\jb{The NFM community will (most likely) not know anything about computational science}

\jb{Overall this paragraph is a little disorganized and needs to be crisper: what are we doing, why are we doing it, how are we doing it?}
Proofs of consistency, stability, and convergence are typically performed by hand, making them prone to possible errors.
Formal verification of mathematical proofs provides a much higher level of confidence of the correctness of manual proofs. Further, formal verification  offers a pathway to leverage mathematical constructs therein, and to extend these proofs to more complex scenarios. 
 
\jb{Some of this needs to move to Related Work. Maybe just keep a mention of Boldo and Taylor-Lagrange here.
Everything else moves to Related Work}

Recently, much effort has  been dedicated to the definition of mathematical structures such as metric spaces, normed spaces, derivatives, limits etc.  in a formal setting using proof assistants such as Coq \cite{o2008certified,boldo2015coquelicot,garillot2009packaging,martin2013certified}. Using automatic provers and proof assistants, a number of works have emerged in the formalization of numerical analysis \cite{boldo2013wave}. Pasca has formalized the properties of the Newton method~\cite{pasca2010formal}. Mayero et al. \cite{mayero2002using} presented a formal proof, developed in the Coq system, of the correctness of an automatic differentiation algorithm. Besides Coq, numerical analysis of ordinary differential equations has also been done in Isabelle/ HOL~\cite{immler2012numerical}. Immler et al.~\cite{Immler,immler2016flow,immler2019flow}, present a formalization of ordinary differential equations and the verification of rigorous (with guaranted error bounds) numerical algorithms in the interactive theorem prover Isabelle/HOL. The formalization comprises flow and Poincar\'e map of dynamical systems. Immler~\cite{10.1007/978-3-319-06200-6_9} implements a functional algorithm that computes enclosures of solutions of ODEs in the interactive theorem prover Isabelle/HOL. In~\cite{brehard2019certificate}, Brehard et al. present a library to verify rigorous approximations of univariate functions on real numbers, with the Coq proof assistant. Brehard~\cite{brehard2019calcul}, worked on rigorous numerics that aims at providing certified representations for solutions of various problems, notably in functional analysis. Work has also been done in formalizing real analysis for polynomials~\cite{cohen2010formalizing}. Boldo and co-workers \cite{boldo2013wave,boldo2014trusting,boldo2010formal} have made important contributions to formal verification of finite difference schemes. They proved consistency, stability and convergence of a second-order centered scheme for the wave equation. 
However, the Lax equivalence theorem -- sometimes referred to as the fundamental theorem of numerical analysis -- which is central to finite difference schemes, has not been formally proven in the general case. 

In this paper, we present a formal proof of the Lax equivalence theorem for a general family of finite difference schemes. We use the definitions of consistency and stability and  prove convergence. To prove the consistency of a second-order centered scheme for the wave equation, Boldo et al.~\cite{boldo2014trusting} made assumptions on the regularity of the exact solution. This regularity is expressed as the existence of Taylor approximations of the exact solution up to some appropriate order. Our formalization instead takes the Taylor--Lagrange theorem of \cite{martin2013certified}, to prove the consistency of a finite difference scheme of any order. It should be noted that the order of accuracy of an explicit finite difference scheme depends on the number of points in the discretized domain (called \textit{stencils}) appearing in the numerical derivative. Our approach is to carry the Taylor series expansion for each of those stencils using the Taylor--Lagrange theorem, and appropriately instantiate the order of the truncated polynomial, to achieve the desired order of accuracy. By incorporating the discretization error into the Lagrange remainder and proving an upper bound for the Lagrange remainder, we propose a rigorous method of proving consistency of a finite difference scheme.

Since the Lax equivalence theorem is an essential tool in the analysis of numerical schemes using finite differences, its formalization in the general case opens the door to the formalization and certification of finite difference-based numerical software.
The present work will enable the formalization of convergence properties for a large class of finite difference numerical schemes, thereby providing formal proofs of convergence properties usually proved by hand, making explicit the underlying assumptions, and increasing the level of confidence in these proofs.

Overall this paper makes the following contributions:
%\jb{List contributions with bullet points}
\begin{itemize}
    \item We provide a formalization in the Coq proof assistant of a general form of the Lax equivalence theorem.
    \item We prove consistency and stability of a second order accurate finite difference scheme for the example differential equation $\frac {d^{2}u}{dx^{2}}=1$.
    \item We formally apply the Lax equivalence theorem on this finite difference scheme for the example differential equation, thereby formally proving convergence for this scheme.
    \item We also provide a generalized framework for a symmetric tri-diagonal (sparse) matrix in Coq. We define its eigen system and provide an explicit formulation of its inverse in Coq. We show that since the symmteric tri-diagonal matrix is normal, one can perform the stability analysis by just uniformly bounding the eigen values of the inverse. This is important because discretizations of mathematical model of physical systems are usually sparse~\cite{KIRK2013217}.
\end{itemize}

This paper is structured as follows.
In Section~\ref{Lax_section}, we review the definitions of consistency, stability and convergence, state the Lax equivalence theorem~\cite{lax1956survey,sanz1985general}, and discuss its formalization in the Coq proof assistant.
In Section~\ref{finite}, we discuss the consistency of a finite difference scheme. In particular, we consider the central difference approximation of the second derivative and formally prove the order of accuracy using the Taylor--Lagrange theorem in the Coq proof assistant. We also relate the pointwise consistency of the finite difference scheme with the Lax equivalence theorem, by instantiating it with an example. In Section~\ref{stability_section}, we discuss the generalized formalization of a symmetric tri-diagonal matrix and later instantiate it with the scheme to prove stability of the scheme. In Section~\ref{Lax_apply}, we apply the Lax equivalence theorem to the concrete finite difference scheme that we are considering. 
In Section~\ref{conclusion}, we conclude by summarizing key takeaways from the paper, and discussing future work.

\section{Lax equivalence theorem}
\label{Lax_section}
In this section, we review the definitions of consistency, stability and convergence, discuss the problem set up  and state the Lax equivalence theorem~\cite{lax1956survey}. 
In this paper and for the formalization, we choose to follow the presentation of Sanz-Serna and Palencia~\cite{sanz1985general}.
We also discuss the proof of the Lax equivalence theorem which is then formalized in the Coq proof assistant.

%\karthik{Mohit: You CANNOT copy full sentences verbatim!!, If  so, keep citing the article every time you do so.}

\subsection{Consistency, Stability and Convergence}

%\jb{Protip: Don't put a space before a forced space \~{}, it looks awful}
\begin{definition}[The Continuous Problem~\cite{sanz1985general}]
Let $X$ (the space of solutions) and $Y$ (the space of data) be normed spaces, both real or both complex. We consider a linear operator $A$ with domain $D \subset X$ and range $R\subset Y$. The problem to be solved is of the form
\begin{equation} \label{true}\small
    Au=f, \qquad f\in Y
\end{equation}
\end{definition}
Here $A$ is not assumed to be bounded, so that unbounded differential operators are included. The problem~(\ref{true}) is assumed to be well-posed, i.e., there exists a \textit{bounded, linear operator}, $E\in B(Y,X)$, such that $EA=I$ in $D$, and that for $f\in Y$, equation (\ref{true}) has a unique solution, $u=Ef$. Furthermore, the solution $u$ depends continuously on the data.
\begin{definition}[The Approximate Problem~\cite{sanz1985general}]
Let $H$ be a set of positive numbers such that $0$ is the unique limit point of $H$. For each $h \in H$ , let $X_{h}, Y_{h}$ be normed spaces and consider the approximate or discretized problem
\begin{equation}\label{approximate}\small
    A_{h}u_{h}=f_{h},\qquad f_{h} \in Y_{h}
\end{equation}
where $A_{h}$ is a linear operator $A_{h}: X_{h}\longrightarrow Y_{h}$.
\end{definition}
We assume that for each $h \in H$, problem (\ref{approximate}) is well-posed and there exists a solution operator, $E_{h}=A_{h}^{-1}$, i.e. $u_{h}=E_{h}f_{h}$. The true solution $u$ and the approximate solution $u_{h}$ can be related with each other by defining a \textit{bounded, linear operator}, $r_{h}: X \to X_{h}$ for each $h \in H$. Similarly, data $f \in Y$ can be related to data in a discrete space, $f_{h} \in Y_{h}$ by defining a restriction operator $s_{h}$. For each $h \in H$, $s_{h}: Y \to Y_{h}$ is also a \textit{bounded, linear operator}.  We assume that the operator norms can be \textcolor{black}{uniformly} bounded:
\begin{equation}\small
    ||r_{h}||\leq C_{1}, \qquad ||s_{h}||\leq C_{2},
\end{equation}
where the constants $C_{1},C_{2}$ are independent of $h$. The true solution $u=Ef$ is compared with the discrete solution $u_{h}=E_{h}s_{h}f$ corresponding to the discretized datum $f$.
The family $(X_{h}, Y_{h}, A_{h},r_{h},s_{h})$ defines a \textit{method} for the solution of (\ref{true})~\cite{sanz1985general}.
\begin{definition}[Convergence~\cite{sanz1985general}]
Let $f$ be a given element in $Y$. The method $(X_{h},Y_{h},A_{h},r_{h},s_{h})$ is convergent for the problem (\ref{true}) if 
\begin{equation}\small \label{convergence}
    \lim_{h \to 0} ||r_{h}Ef-E_{h}s_{h}f||_{X_{h}}=0
\end{equation}
We say that the method is convergent if it is convergent for each problem (\ref{true}) for any $f$  in $Y$.
\end{definition}

%\mohit{ added some text here}
Intuitively, this means that in the limit of the discretization step, $h$, tending to zero, the numerical solution $E_{h}s_{h}f$ approaches the analytical solution $r_{h}Ef$. The analytical solution $r_{h}Ef$ is the restriction of the true (analytical) solution, $u=Ef$, onto the grid of size $N=1/h$, and $E_{h}s_{h}f$ is the discrete solution, $u_{h}=E_{h}f_{h}$ computed on the grid of size $N$. 

\begin{definition}[Consistency~\cite{sanz1985general}]
Let $u$ be a given element in $D$. The method is consistent at $u$ if 
\begin{equation}\small \label{consistency}
    \lim_{h \to 0} ||A_{h}r_{h}u - s_{h}Au||_{Y_{h}} = 0
\end{equation}
A method is consistent if it is consistent at each $u$ in a set $D_{o}$ such that the image $A(D_{o})$ is dense in $Y$.
\end{definition}

%\mohit{text added here}
Intuitively, this means that in the limit of the discretization step, $h$, tending to zero, the finite difference scheme $A_{h}u_{h}=f_{h}$ approaches the differential equation $Au=f$, i.e., we are discretizing the right differential equation. 

\begin{definition}[Stability~\cite{sanz1985general}]
The method is stable if there exists a constant $K$ such that
\begin{equation}\small \label{stability}
    ||E_{h}||_{B(Y_{h},X_{h})} \leq K
\end{equation}
\end{definition}
%

%\mohit{text added here}
Intuitively, stability of the numerical scheme means that a small numerical perturbation does not allow the solution to blow up. Uniform boundedness of the inverse $E_{h}=A_{h}^{-1}$ is a check on the conditioning of matrices (sensitivity to small perturbations), i.e., it ensures that the matrix $A_{h}$ is not ill-conditioned. Thus, if the numerical problem~(\ref{approximate}) were unstable, even though we were trying to solve the right differential equation, we would never converge to the true solution. Hence, both stability and consistency are sufficient for proving convergence of the numerical scheme.

The quantities within the norms (\ref{convergence}) and (\ref{consistency}) are, respectively, the \textit{global} and \textit{local} discretization errors.

\begin{theorem}[Lax equivalence theorem~\cite{sanz1985general}]\label{Lax}
Let
$(X,Y,A,X_{h},Y_{h},A_{h},r_{h},s_{h})$ be as above. If the method is consistent and stable, then it is convergent.
\end{theorem}
\begin{proof}\small
We start  with the definition of \textit{convergence} in (\ref{convergence}),
\begin{align*}\small
     & \lim_{h \to 0}||r_{h}Ef-E_{h}s_{h}f||_{X_{h}}\\
     &= \lim_{h \to 0}||r_{h}u-E_{h}s_{h}f||_{X_{h}}\quad (u\overset{\Delta}{=} Ef)\\
        & = \lim_{h \to 0}||r_{h}u-E_{h}s_{h}Au||_{X_{h}}\quad (f \overset{\Delta}{=} Au)\\
        & = 
        \lim_{h \to 0}||Ir_{h}u-E_{h}s_{h}Au||_{X_{h}}\quad(r_{h}u=Ir_{h}u)\\
        & =
        \lim_{h \to 0}||E_{h}A_{h}r_{h}u-E_{h}s_{h}Au||_{X_{h}}\quad (E_{h}A_{h}\overset{\Delta}{=}I)\\
        & \leq \lim_{h \to 0}||E_{h}||_{B(Y_{h},X_{h})}||(A_{h}r_{h}u-s_{h}Au)||_{Y_{h}} \\
        & \leq K \lim_{h \to 0}||(A_{h}r_{h}u-s_{h}Au)||_{Y_{h}} \quad (\text{From stability: } (\ref{stability}))\\ 
        & = 0 \quad (\text{From Consistency:  } (\ref{consistency})) 
\end{align*}
\end{proof}
\subsection{Formalization in the Coq Proof Assistant}
%\jb{This is still quite detailedThis is way too detailed and at the same time difficult to read, because you don't give the intuition. What is important is to give the high-level picture, not to give all the tiny details (especially since Coq already checks all the tiny details).
%See Boldo et al.'s papers for some examples on how this is done.}
%\mohit{Provided a detailed description along with main definitions while laying out all important details required for formalizing the proof}
In this Section we show how we formalized the proof of the Lax equivalence theorem \cite{sanz1985general} in the Coq proof assistant.
All of the Coq formal proofs mentioned in this paper, containing the proofs of consistency, stability and convergence of finite difference schemes, and of the Lax equivalence theorem, are available at  \url{http://www-personal.umich.edu/~jeannin/papers/NFM21.zip}.

The \texttt{Coquelicot} library~\cite{boldo2015coquelicot,Coquelic9:online} defines mathematical structures required for implementing the proof. \textcolor{black}{Since we use Coquelicot and standard reals library which are based on classical axiomatization of reals, our proofs are also non-constructive~\cite{boldo2015coquelicot}.} We define the \textit{Banach spaces} (complete normed spaces, complete in the metric defined by the norm \cite{kreyszig1978introductory}) $(X,Y,X_{h},Y_{h})$ using a canonical structure, \texttt{CompleteNormedModule}, in Coq \cite{garillot2009packaging}. 

%\jb{Unnecessarily detailed, can save space}
The definitions of the true problem (\ref{true}) and the approximate problem (\ref{approximate}) require that the mappings $A: X \to Y $ and $A_{h}: X_{h} \to Y_{h}$ be linear, and the solution operators $E: Y \to X$ and $E_{h}: Y_{h} \to X_{h}$ be linear and bounded. The linear mappings $A_{h}$ and $E_{h}$ are defined as functions of $h \in \mathbb{R}$.
Boldo et al.~\cite{boldo2017coq} have defined linear mapping in the context of a \texttt{ModuleSpace} and bounded linear mapping in the context of a \texttt{NormedModule} in their formalization of the \textit{Lax Milgram Theorem} in Coq~\cite{httpswww42:online,FlorianF30:online}. We extended these definitions in the context of \texttt{CompleteNormedModule}. 

The definition of \textit{consistency} (\ref{consistency}) and \textit{convergence} (\ref{convergence}) hold in the limit of $h$ tending to zero. Thus, an important step in the proof is to express these limits in Coq. Formally, the notion of $f$ tending to $l$ at the limit point $x$ requires, for any $\epsilon > 0$, to find a neighborhood $V$ of $x$ such that any point $u$ of $V$ satisfies $|f(u)-l|<\epsilon$ \cite{boldo2015coquelicot}. This notion has been formalized in \texttt{Coquelicot} \cite{Coquelic9:online} using the concept of \textit{filters}. In topology, a filter is a set of sets, which is nonempty, upward closed, and closed under intersection \cite{cohen2017formal}. It is commonly used to express the notion of convergence in topology. We have used a filter, \texttt{locally x} \cite{lelay2015express} to denote an open neighborhood of $x$, and predicate \texttt{filterlim} \cite{lelay2015express} to formalize the notion of convergence (in the context of limits) of $f$ towards $l$ at limit point $x$, i.e. $\lim_{x \to a} f(x) =l$. Therefore, the definition of consistency (\ref{consistency}) is expressed as:
\begin{small}
\begin{verbatim}
(is_lim (fun h:R => norm (minus (Ah h (rh h u)) (sh h (A u)))) 0 0
\end{verbatim}
\end{small}
where the limits of functions is expressed using \texttt{is\_lim}~\cite{boldo2015coquelicot}.

We next discuss the formalization of the statement of convergence of a finite difference scheme in Coq. We note that from Theorem~\ref{Lax}, \textit{consistency} and \textit{stability} imply \textit{convergence}. This notion is expressed in Coq as follows:
\mohit{Hopefully, this is addressed in the following paragraphs}
\jb{present math first, Coq next}
\begin{small}
\begin{verbatim}
(is_lim (fun h:R => norm (minus (Ah h (rh h u)) (sh h (A u))))  0 0 
    (*Consistency*) /\ 
(exists K:R , forall (h:R), operator_norm(Eh h)<=K ) (* Stability*) -> 
is_lim(fun h:R=>norm (minus (rh h (E(f))) (Eh h (sh h (f))))) 0 0)
    (*Convergence*).
\end{verbatim}
\end{small}
where the \textit{operator norm} is defined as $||f||_{\phi}=sup_{u \neq 0_{E}\land \phi(u)}\frac {||f(u)||_{F}}{||u||_{E}}$ and has been formally defined in \cite{boldo2017coq}.

The basic idea is that we  bound the  \textit{global discretization error} ($||r_{h}Ef - E_{h}s_{h}f||)$  above using the stability criterion, i.e. $||r_{h}Ef-E_{h}s_{h}f|| \leq K  ||A_{h}r_{h}u - s_{h}Au||$,  and then prove that as the \textit{local discretization error} ($||A_{h}r_{h}u - s_{h}Au||)$ tends to zero in the limit of $h$ tending to zero, the upper bound on the global discretization error tends to zero (using the property of limits). Using the property of norm , i.e. $0 \leq ||r_{h}Ef- E_{h}s_{h}f||$, we arrive at the inequality
\begin{equation*}\small
    0 \leq ||r_{h}Ef- E_{h}s_{h}f|| \leq K ||A_{h}r_{h}u - s_{h}Au||
\end{equation*}
In Coq, we define the lower bound of the inequality as a constant function with value $0$ as: \texttt{fun\;\_ => 0}.
Since 
the limit of a constant function is the constant itself, i.e. $\lim_{h \to 0} 0 =0$, 
and 
$lim_{h\to 0}||A_{h}r_{h}u-s_{h}Au|| = 0$ (Consistency), using the \textit{sandwich theorem} for limits,
$\lim_{h\to 0}||r_{h}Ef- E_{h}s_{h}f||=0$. The \textit{sandwich theorem} states that if we have functions obeying the inequality: $f(x)\leq g(x) \leq h(x)$ and $\lim_{x \to a}f(x)=L \quad \land \quad \lim_{x \to a}h(x)=L$ on some open neighborhood of $x=a$ , then $\lim_{x \to a}g(x)=L$. This proves the convergence of Definition~\ref{convergence} and completes the proof of the Lax equivalence theorem.

\section{Proof of consistency of a sample finite difference scheme}
%\jb{Are you focusing on one scheme in particular?}
%\jb{This section is about the consistency of finite central difference scheme for second order derivative}
\label{finite}

A finite difference scheme (FD) approximates a differential equation with a difference equation. The derivatives are expressed in terms of function values at finite number of points in the dicretized domain. For instance, consider a simple differential equation, $\frac {d^{2} u}{d x^{2}}=1$ on a domain $x \in (0,L)$ with boundary conditions $u(0)=0$ and $u(L)=0$, where L is the length of the domain. A second order accurate finite difference approximation would be $\frac {u(x+\Delta x)-2u(x)+u(x-\Delta x)}{\Delta x ^2}=1$, where $\Delta x$ is the discretization step and $x$ is the point at which the difference equation is evaluated. We will refer to this as numerical scheme $\mathcal{N}_h$. Since we are computing a numerical approximation to the actual derivatives, we are interested in knowing the order of the discretization error.
\begin{definition}[Discretization error]\small
    Let $D(u)$ denote the true derivative of a function $u:\mathbb{R} \to \mathbb{R}$ and $N(u)$ denote the finite difference approximation of the true derivative. The discretization error (commonly referred to as the truncation error) ($\tau$) is then defined as:
    \begin{equation}
        \tau \overset{\Delta}{=} D(u)-N(u)
    \end{equation}
\end{definition}
If the function $u$ is \textit{analytic}, it can be expressed as a \textit{Taylor series expansion} at the point of evaluation. The truncation error is then evaluated by expressing the numerical derivatives in terms of a truncated Taylor polynomial and then taking a difference of the true derivative and the numerical derivative. This gives us an upper bound on the discretization error. If a numerical method is consistent, the truncation error can be expressed as:
%\jb{What does $\sim$ mean? This is formal mathematics, not handwaving.}
\begin{equation*}\small
    \tau = \mathcal{O}(\Delta x ^{n})
\end{equation*}
when $\Delta x$ tends to zero, and where $n$ is the order of the truncated Taylor polynomial. We use this idea to formalize the proof of consistency of a finite difference scheme. This requires the use of an important theorem from calculus, the Taylor--Lagrange theorem.

\begin{theorem}[Taylor--Lagrange theorem]\label{Taylor_Lagrange}\small
    Suppose that $f$ is $n+1$ times differentiable on some interval containing the center of convergence $c$ and $x$, and let $P_{n}(x)= f(c)+\frac {f^{(1)}(c)}{1!}(x-c)+\frac{f^{2}(c)}{2!}(x-c)^{2}+..+\frac{f^{(n)}(c)}{n!}(x-c)^{n}$ be the $n^{th}$ order Taylor polynomial of $f$ at $x=c$. Then $f(x)=P_{n}(x)+E_{n}(x)$ where $E_{n}(x)$ is the error term of $P_{n}(x)$ from $f(x)$. i.e. $E_{n}=f(x)-P_{n}(x)$, and for $\xi$ between $c$ and $x$, the Lagrange remainder form of the error $E_{n}$ is given by the formula $E_{n}(x)=\frac{f^{n+1}(\xi)}{(n+1)!} (x-c)^{(n+1)}$.
\end{theorem}
Martin-Dorel et al. \cite{martin2013certified} proved the Taylor--Lagrange theorem formally in Coq, and it is available in the \texttt{Coq.Interval} library \cite{Interval66:online,brisebarre2012rigorous}. We used this formalization of the Taylor--Lagrange theorem to prove the consistency of a finite difference scheme.

We will specifically prove that for a central difference approximation of the second derivative, $\frac {d^{2}u}{dx^{2}}$, expressed as : $\frac {u(x+\Delta x)-2 u(x)+u(x-\Delta x)}{(\Delta x)^{2}}$, the truncation error $\tau$ is quadratic in $\Delta x$: %\jb{Make it a theorem}
\begin{equation*}\small
    \tau = \left | \frac{d^{2}u}{dx^{2}}- \frac {u(x+\Delta x)-2u(x)+u(x-\Delta x)}{(\Delta x)^2} \right | = \mathcal{O}(\Delta x ^2)
\end{equation*}

\subsection{Proof of consistency for the finite difference scheme}\label{point_const}
 We want to prove that for a central difference approximation of the second derivative in the numerical scheme  $\mathcal{N}_h$, the truncation error, $\tau= \mathcal{O}(\Delta x^2)$.
 By invoking the definition of Big-O notation, the theorem statement can be stated as:
 \begin{equation}\label{cons_2}\small 
 \exists \gamma >0, \Gamma>0, \left | \frac{d^{2}u}{dx^{2}} - \frac { u(x+\Delta x)-2u(x)+ u(x-\Delta x)}{(\Delta x)^2}\right |\leq \Gamma (\Delta x ^2), \; 0<|\Delta x|<\gamma.  
 \end{equation}
The equation (\ref{cons_2}) is stated formally in Coq as:
\begin{small}
 \begin{verbatim}
Theorem taylor_FD (x:R): Oab x ->exists gamma:R, gamma >0 /\ exists G:R, 
G>0/\ forall dx:R, dx>0 -> Oab (x+dx) -> Oab (x-dx)->(dx< gamma ->
Rabs((D 0 (x+dx)- 2*(D 0 x) + D 0 (x-dx))*/(dx * dx)- D 2 x)<= G*(dx^2)).
\end{verbatim}
\end{small}where \texttt{Oab x} mean $a < x < b$ and \texttt{D k x} denotes $k^{th}$ derivative of $u$ with respect to x. \\
We start by introducing the following lemmas required to complete the proof.

\begin{small}
\begin{lemma}
[$|F(x)|\sim \mathcal{O}(\Delta x)^4$]
\label{lem:lem1}
    $\forall x \in (a,b),\exists\; \eta \in \mathbb{R}, \eta>0 \land \exists \;M \in \mathbb{R}, M>0 \land\\ \forall \Delta x \in \mathbb{R}, \Delta x >0
    \to (x+\Delta x) \in (a,b) \to \Delta x < \eta \to |F(x)|\leq M(\Delta x)^4.$
\end{lemma}
Here, $F(x)$ is the Lagrange remainder in the expansion of $u(x+\Delta x)$ up to degree 3 and is defined as:
\begin{equation}\label{def_1}
 F(x) \overset{\Delta}{=} u(x+\Delta x)-u(x)-\Delta x \frac{du}{dx}\Big|_x -\frac{1}{2!}(\Delta x)^{2}\frac{d^{2}u}{dx^2}\Big|_x-
\frac{1}{3!}(\Delta x)^{3}\frac{d^{3}u}{dx^{3}}\Big|_x    
\end{equation}
Thus, Lemma~\ref{lem:lem1} states that the Lagrange remainder $F(x)= \frac{1}{4!}(\Delta x)^4 \frac{d^4 u(\xi)}{dx^4}$ is  of order $(\Delta x)^4$ for all $\xi \in (x, x+\Delta x)$.

\begin{lemma}
[$|G(x)|\sim \mathcal{O}(\Delta x)^4$]
\label{lem:lem2}
    $\forall x \in (a,b), \exists\; \delta \in \mathbb{R}, \delta>0 \land \exists \;K \in \mathbb{R}, K>0 \land \\
    \forall \Delta x \in \mathbb{R}, \Delta x > 0 \to (x-\Delta x) \in (a,b)\to \Delta x < \delta \to |G(x)|\leq K(\Delta x)^4.$
\end{lemma}
\end{small}
Here, $G(x)$ is the Lagrange remainder in the expansion of $u(x-\Delta x)$ up to degree 3 and is defined as:
\begin{equation}\label{def_2}\small
G(x) \overset{\Delta}{=}u(x-\Delta x)-u(x)+\Delta x \frac{du}{dx}\Big|_x-\frac{1}{2!}(\Delta x)^{2}\frac{d^{2}u}{dx^{2}}\Big|_x+\frac{1}{3!}(\Delta x)^{3}\frac{d^{3}u}{dx^{3}}\Big|_x
\end{equation}
Thus, Lemma~\ref{lem:lem2} states that the Lagrange remainder $G(x)= \frac{1}{4!}(\Delta x)^4 \frac{d^4 u(\xi)}{dx^4}$ is  of order $(\Delta x)^4$ for all $\xi \in (x-\Delta x, x)$.

Both the lemmas are a straightforward  application of the Taylor--Lagrange theorem (Theorem~\ref{Taylor_Lagrange}), and are crucial to the formalization of the proof of the consistency of the finite difference scheme.

Next, we present an informal proof of the theorem followed by a discussion on the formal proof of the consistency theorem.
\begin{proof}\small
%\jb{you mean \ref{lem:lem1}?}
\begin{equation}\label{lemma_1}\small
    |F(x)|\leq M (\Delta x )^4 \quad \text{[From Lemma~\ref{lem:lem1}]}
\end{equation}
%\jb{you mean \ref{lem:lem2}?}
\begin{equation}\label{lemma_2}\small
    |G(x)| \leq K (\Delta x)^4 \quad \text{[From Lemma~\ref{lem:lem2}]}
\end{equation}
Adding equation (\ref{lemma_1}) and (\ref{lemma_2}), we get:
\begin{align}\small
 &|F(x)|+|G(x)| \leq (M+K) (\Delta x)^{4} \nonumber \\
\implies&|F(x)+ G(x)| \leq (M+K) (\Delta x)^{4} \nonumber \\
&\text{[Using the triangle inequality, $(|F(x)+G(x)| \leq |F(x)|+|G(x)|)$ ]}\nonumber \\
\implies& |F(x)+G(x)| \leq \Gamma (\Delta x)^ {4} \quad(\text{Instantiating} \Gamma := M+K) \label{cons_res}
\end{align}
Unfolding the definitions $F(x)$ and $G(x)$, and doing the algebra we get:
\begin{align}
&\Big |u(x+\Delta x)-2u(x)+u(x-\Delta x)-(\Delta x)^{2}\frac{d^{2}u}{dx^{2}}\Big| \leq \Gamma (\Delta x ^4) \nonumber\\
\implies &\Big |\frac{u(x+\Delta x)-2u(x)+u(x-\Delta x)}{(\Delta x)^{2}}-\frac{d^{2}u}{dx^{2}}\Big| \leq \Gamma (\Delta x ^2)\label{final_FD} \quad \textbf{[QED]}
\end{align}
\end{proof}
An important point to note is that the condition $|F(x)|+|G(x)|\leq M(\Delta x)^4+K(\Delta x)^4$ holds when $0<|\Delta x|<\gamma$, where $\gamma$ is as defined in (\ref{cons_2}). We therefore choose, $\gamma = min(\eta, \delta)$, where $\eta$ is such that, $|F(x)|\leq M(\Delta x)^4$ holds when $0<|\Delta x|<\eta$, and $\delta$ is such that, $|G(x)|\leq K(\Delta x)^4$ holds when $0<|\Delta x|<\delta$.

\subsection{Formalization in the Coq Proof assistant}
We followed the proof above and formalized it in the Coq proof assistant.
To apply  the Taylor--Lagrange theorem \cite{martin2013certified}  to the consistency analysis of a central difference approximation, we broke down the theorem statement into two lemmas as discussed in the previous section. Therefore, in this section, we will discuss the proof of Lemma~\ref{lem:lem1} and~\ref{lem:lem2}.
\subsubsection{Proof of Lemma~\ref{lem:lem1}:}
Formally Lemma~\ref{lem:lem1} %\jb{you mean \ref{lem:lem1}? Please use labels and references everywhere, it avoids those confustions.} 
is stated in Coq as:
\begin{small}
\begin{verbatim}
Lemma taylor_uupper (x:R): Oab x-> exists eta: R, eta>0 /\ 
    exists M :R, M>0  /\ forall dx:R, dx>0 -> Oab (x+dx) -> 
    (dx<eta -> Rabs(D 0 (x+dx)- Tsum 3 x (x+dx))<=M*(dx^4)).
\end{verbatim}
\end{small}
In the proof of the Lemma, existential quantification associated with $\eta$ and $M$ has to be addressed. We chose $\eta$ as $b-x$, since the interval in which we are studying Taylor--Lagrange for $u(x+\Delta x)$ is $[x,b]$. Since $\Delta x \in (x,b)$ and $\Delta x < \eta$, it seems logical to chose $\eta=b-x$. For the choice of $M$, we obtained extreme bounds in the interval. Since the function $u$ and its derivatives are continuous in a compact set $[x,b]$, we are guaranteed to get maximum and minimum values. In Coq, we applied the lemma \texttt{continuity\_ab\_max} to obtain a maximum value, $\left(\frac{d^{4}u}{dx^{4}}\right)_{max}=\frac{d^{4}u(F)}{dx^{4}}$ such that $\frac{d^{4}u(\xi)}{dx^{4}}\leq \frac{d^{4}u(F)}{dx^{4}}, \forall \xi \in [x,b]$. Similarly, we apply the lemma \texttt{continuity\_ab\_min} to obtain a minimum value, $\left(\frac{d^{4}u}{dx^{4}}\right)_{min}=\frac{d^{4}u(G)}{dx^{4}}$ such that $\frac{d^{4}u(G)}{dx^{4}}\leq \frac{d^{4}u(\xi)}{dx^{4}}, \forall \xi \in [x,b]$. \\
Thus, $M$ is chosen as $M=max\left(\left|\frac{d^{4}u(G)}{dx^{4}}\right|,\left|\frac{d^{4}u(F)}{dx^{4}}\right|\right)$. With this choice of $M$, we can bound the Lagrange remainder or the trunction error from above and thus prove Lemma~\ref{lem:lem1}. 
\subsubsection{Proof of Lemma~\ref{lem:lem2}:}
Formally Lemma~\ref{lem:lem2} %\jb{you mean \ref{lem:lem2}?} 
is stated in Coq as:
%\jb{Indent this better so it's mor readable}
\begin{small}
\begin{verbatim}
Lemma taylor_ulower (x:R): Oab x -> exists delta: R, delta>0 /\ 
    exists K :R, K>0 /\ forall dx:R, dx>0 ->Oab (x-dx) -> 
    (dx<delta -> Rabs(D 0 (x-dx)-Tsum 3 x (x-dx))<=K*(dx^4)).
\end{verbatim}
\end{small}
The proof of Lemma~\ref{lem:lem2} follows the same approach as that of Lemma~\ref{lem:lem1}. Here, we chose $\delta $ as $x-a$, since the interval in which we are studying Taylor--Lagrange theorem for $u(x-\Delta x)$, $\Delta x \in (a,x)$, and $\Delta x < \delta$. We chose $K$ in the same way as we chose $M$ in Lemma~\ref{lem:lem1} except that the interval in which we obtain maximum and minimum values for $\frac{d^{4}u}{dx^{4}}$ is $[a,x]$ in this case. Thus, $\left(\frac{d^{4}u}{dx^{4}}\right)_{min}=\frac{d^{4}u(G)}{dx^{4}}$, $\left(\frac{d^{4}u}{dx^{4}}\right)_{max}=\frac{d^{4}u(F)}{dx^{4}}$,and $K=max\left(\left|\frac{d^{4}u(G)}{dx^{4}}\right|,\left|\frac{d^{4}u(F)}{dx^{4}}\right|\right), \forall c \in [a,x]$.

To prove the main theorem statement on consistency, we break the statement into Lemma~\ref{lem:lem1} and~\ref{lem:lem2}, by instantiating $\Gamma = M+K$, and $\gamma = \min(\eta, \delta)$, where $(M,\eta)$ and $(K,\delta)$ have been defined as in Lemma~\ref{lem:lem1} and~\ref{lem:lem2} respectively, in the manner shown in section~(\ref{point_const}). To implement this instantiation, we have to carefully \textit{destruct} the lemmas introduced in the theorem statement. Then, we simply apply lemma~\ref{lem:lem1} and~\ref{lem:lem2}, to complete the main proof.

%\jb{Return on experience: how long did it take? What was difficult? What took the longest? What was surprising? Which things would you do differently?}

\subsection{Relating  pointwise consistency to the Lax equivalence theorem}
%\jb{apply consistency of Section (\ref{point_const}) to prove consistency of equation $\frac{d^{2}u}{dx^2}=1$. We never prove stability}
In this section, we relate the proof of consistency from Section~\ref{point_const} with the Lax equivalence Theorem~\ref{Lax}. The numerical discretization of the differential equation can be expressed in the discrete domain as:
\begin{equation}\label{FD_scheme}
\small
    \underbrace{
    \frac{1}{h^{2}}
    \begin{bmatrix}
    1 & 0 & 0 & 0 & \hdots & 0\\
    1 &-2 & 1& 0 & \hdots & 0\\
    \vdots& \ddots &\ddots&\ddots& &\vdots\\
    0 & \hdots & 1 & -2 & 1 & 0\\
    0& \hdots & 0 & 1 & -2 & 1\\
     0 &\hdots & 0 & 0 & 0  &1
    \end{bmatrix}
    }_\text{$A_{h}$}
    \underbrace{
    \begin{bmatrix}
    u_{o}\\
    u_{1}\\
   \vdots\\
    u_{N-2}\\
    u_{N-1}\\
    u_{N}
    \end{bmatrix}
    }_\text{$r_{h}u$}=
    \underbrace{
    \begin{bmatrix}
    0\\
    1\\
    \vdots\\
    1\\
    1\\
    0
    \end{bmatrix}
    }_\text{$s_{h}Au$}
\end{equation}
Comparing with the statement of consistency (\ref{consistency}), we have
\begin{small}
\begin{equation}\label{FD1}
\small
    \lim_{h \to 0}\left|\left |
   \frac{1}{h^{2}}
    \begin{bmatrix}
    1 & 0 & 0 & 0 & \hdots & 0\\
    1 &-2 & 1& 0 & \hdots & 0\\
    \vdots& \ddots &\ddots&\ddots& &\vdots\\
    0 & \hdots & 1 & -2 & 1 & 0\\
    0& \hdots & 0 & 1 & -2 & 1\\
     0 &\hdots & 0 & 0 & 0  &1
     \end{bmatrix}
    \begin{bmatrix}
    u_{o}\\
    u_{1}\\
    \vdots\\
    u_{N-2}\\
    u_{N-1}\\
    u_{N}
    \end{bmatrix}-
    \begin{bmatrix}
    0\\
    1\\
    \vdots\\
    1\\
    1\\
    0
    \end{bmatrix}
    \right| \right|
     = \lim_{h \to 0}\left|\left|
    \begin{bmatrix}
    \frac{u_{o}}{h^2}\\
    \frac{u_{o}-2u_{1}+u_{2}}{h^{2}}-1\\
    \frac{u_{1}-2u_{2}+u_{3}}{h^{2}}-1\\
    \vdots\\
    \frac{u_{N-2}-2u_{N-1}+u_{N}}{h^{2}}-1\\
    \frac{u_{N}}{h^2}
    \end{bmatrix}
    \right| \right|=0
\end{equation}
\end{small}
Taking the vector norm in the $L_{1}$ sense, $||.||_{1}$, equation (\ref{FD1}) can be written as:
\begin{equation}
\lim_{h \to 0} \Big[\left|\frac{u_{o}}{h^2}\right|+\left|\frac{u_{o}-2u_{1}+u_{2}}{h^{2}}-1\right| +..+ \left|\frac{u_{N-2}-2u_{N-1}+u_{N}}{h^{2}}-1\right|+ \left|\frac{u_{N}}{h^2}\right|\Big]=0\label{term_1}
\end{equation}
$\lim_{h \to 0} \frac{u_{o}}{h^2}=0$ and $\lim_{h \to 0} \frac{u_{N}}{h^2}=0$, trivially because of the boundary conditions we imposed, i.e. $u_{o}=0$ and $u_{N}=0$. \textcolor{black}{The norm used in~(\ref{FD1}) are in the space $Y_h$, i.e., $||.||_{Y_h}$}.\\
This reduces to proving:
\begin{equation} \label{FD2}\small
  \sum_{i=1}^{N-1}  \lim_{h \to 0} \left| \frac {u_{i-1}-2u_{i}+u_{i+1}}{h^{2}}-1\right|=0
\end{equation}
But from the Taylor--Lagrange analysis discussed in section~(\ref{point_const}), we have
\begin{equation}\label{FD3}\small
    \left |\frac {u_{i-1}-2u_{i}+u_{i+1}}{h^{2}}- \frac{d^{2}u}{dx^{2}}\Big|_{x_i} \right| \leq Ch^2 
\end{equation}
where $C$ is a constant, and $u_{i}=u(x_i),  u_{i-1}=u(x_i -h), u_{i+1}=u(x_i +h)$. Substituting $\left. \frac{d^{2}u}{dx^{2}} \right|_{x_{i}}=1$, and using the inequality (\ref{FD3}) and equation(\ref{FD2}), we get
\begin{equation}\small
    \sum_{i=1}^{N-1} 0 \leq \sum_{i=1}^{N-1}  \lim_{h \to 0} \left| \frac {u_{i-1}-2u_{i}+u_{i+1}}{h^{2}}-1\right| \leq \sum_{i=1}^{N-1} \lim_{h \to 0} |C h^2|
\end{equation}
But, 
  $ \sum_{i=1}^{N-1} \lim_{h \to 0} |C h^2|=0$. Hence, using the sandwich theorem, we prove that
  \begin{small}
\begin{equation*}
      \sum_{i=1}^{N-1}  \lim_{h \to 0} \left| \frac {u_{i-1}-2u_{i}+u_{i+1}}{h^{2}}-1\right|=0 \qquad \textbf{[QED]}
\end{equation*}
\end{small}
\subsection{Formalization in Coq}
In order to represent, $x_{i},\; i=0..N$, 
we define $x$ of type: \texttt{nat $\to$ R}. %i.e.,
%\begin{small}
%\begin{verbatim}
 %   Variable x:nat -> R.
%\end{verbatim}
%\end{small}
The boundary conditions are imposed as hypothesis statements:
\begin{small}
\begin{verbatim}
    Hypothesis u_0 : (D 0 (x 0))= 0.
    Hypothesis u_N: (D 0 (x N)) =0.
\end{verbatim}
\end{small}
The differential equation is defined as:
\begin{small}
\begin{verbatim}
Hypothesis u_2x: forall i:nat, (D 2 (x i)) =1.
\end{verbatim}
\end{small}

Equation (\ref{FD2}) is formalized as a lemma statement:
\begin{small}
\begin{verbatim}
Lemma lim_sum:is_lim (fun h:R => 
  sum_n_m (fun i:nat =>Rabs (( D 0 (x i -h) -2* (D 0 (x i)) 
    + D 0 (x i +h))*/(h^2) -1)) 1%nat (pred N)) 0 0.
\end{verbatim}
\end{small}
This is where we integrate the proof of pointwise consistency of the FD scheme from section (\ref{finite}).

The main theorem statement which is an application of the statement of consistency required in the proof of Lax equivalence theorem from section (\ref{Lax_section}) is as follows:
\begin{small}
\begin{verbatim}
Theorem consistency_inst: forall (U:X) (f:Y) (h:R) (uh: Xh h)
 (rh: forall (h:R), X -> (Xh h)) (sh: forall (h:R), Y->(Yh h))
 (E: Y->X) (Eh:forall (h:R),(Yh h)->(Xh h)), 
 is_lim (fun h:R => norm (minus (Ah h (rh h U)) (sh h (A U)))) 0 0.
\end{verbatim}
\end{small}
We note here that the above-mentioned formalization is not unique to the second order scheme that we discussed. The approach we discuss can easily be generalized to verify consistency of any finite difference scheme. The crucial step in such a generalization is the appropriate instantiation of the $A_{h}$ matrix and the vectors $r_{h}u$ and $s_{h}Au$.
  
\section{Stability of the scheme}\label{stability_section}
In this section we discuss the stability of the scheme  $\mathcal{N}_h$. From section~\ref{Lax_section}, stability of a numerical scheme requires the solution operator $E_{h}=A_{h}^{-1}$ to be uniformly bounded. We prove this by bounding the eigenvalues of $E_{h}$ uniformly. Eigenvalues of $E_{h}$ are just inverse of the eigenvalues of $A_{h}$. A formal proof of this can be referred to in the Appendix~\ref{inverse_spectrum}.

We will first discuss a generalized framework for the formalization of stability for a symmetric tri-diagonal matrix in Coq. We denote this matrix with $A_{h}(a,b,c)$ with $c=a$ for symmetry. This notation means that $b$ is on the diagonal, $c$ is on the upper diagonal and $a$ is on the lower diagonal. All the other entries are zero. Since we are treating stability from a spectral viewpoint, we next discuss the formalization of the Eigen system for $A_{h}(a,b,a)$.

\subsection{Lemma to verify that the eigenvalues and eigenvectors belong to the spectrum of $A_{h}(a,b,a)$}
Analytical expressions for the eigenvalues and eigenvectors of $A_{h}(a,b,c)$ are given by:
\begin{small}
\begin{equation*}
\lambda_{m}=b+2\sqrt{ac}\cos{\left[\frac{m\pi}{N+1}\right]}; \quad   s_{m}=\left(s_{j}\right)_{m}=\left[\frac{a}{c}\right]^{j-1/2}\sqrt{\frac{2}{N+1}}\sin{\left[j \frac{m\pi}{N+1}\right]}\;
\end{equation*}
\end{small} $ \forall m,j = 1..N$.
In Coq, we defined $\lambda_m$ and $s_m$ as follows:
\begin{small}
\begin{verbatim}
Definition Eigen_vec (m N:nat) (a b c:R):= mk_matrix N 1%nat (fun i j => 
    sqrt ( 2 / INR (N+1))*(Rpower (a */c) (INR i +1 -1*/2))* 
        sin(((INR i +1)*INR(m+1)*PI)*/INR (N+1))).

Definition Lambda (m N:nat) (a b c:R):= mk_matrix 1%nat 1%nat (fun i j => 
    b + 2* sqrt(a*c)* cos ( (INR (m+1) * PI)*/INR(N+1))).
\end{verbatim}
\end{small}Since naturals in Coq start with 0, we write \texttt{INR (m+1)} and \texttt{INR i+1}.  

We then formally verify that the analytical expressions for the pair $(\lambda_{m}, s_{m})$ indeed belong to the spectrum of $A_{h}$. From now on, we will refer to $A_{h}(a,b,a)$ as $A_{h}$ for the sake of brevity.  
In Coq, we state this formally as:
\begin{small}
\begin{verbatim}
Lemma eigen_belongs (a b c:R): forall (m N:nat), (2 < N)%nat -> 
    (0 <= m < N)%nat -> a=c /\ 0<c-> (LHS m N a b c) = (RHS m N a b c).
\end{verbatim}
\end{small}
where, $LHS\overset{\Delta}{=}A_{h}s_{m}$ and $RHS\overset{\Delta}{=}s_{m}\lambda_{m}$. Here we used the definition of eigenvalue-eigenvector, i.e., $A_{h} s_{m}\overset{\Delta}{=}\lambda_{m}s_{m}$. Formalizing the proof of the lemma \texttt{eigen\_belongs} was challenging due to the structure of the matrix $A_{h}$. $A_{h}$ is a tri-diagonal matrix with non-zero entries on the diagonal, sub-diagonal and super-diagonal. The other entries are zero and hence the matrix is sparse.
\begin{equation}\small
\label{sparse_sum}
\therefore\;
\underbrace{
    \sum_{j=0}^{N-1} {A_{h} (i,j) s_{m}(i)}
    }_\text{$A_{h}(i,j) \neq 0$} + 
    \underbrace{
    \sum_{j=0}^{N-1} {A_{h} (i,j) s_{m}(i)}
    }_\text{$A_{h}(i,j) =0$} = \lambda_{m} s_{m}(i); \quad 0\leq i \leq N-1
\end{equation}
In Coq, we have to carefully destruct the matrix $A_{h}$ to separate the non-zero and zero sums in the LHS of equation~(\ref{sparse_sum}). The idea is to do a case analysis on the row-index $i$, and has been illustrated in figure~(\ref{tridiagonal}) in the Appendix~\ref{lemma_eigen}. Details on the formal proof of the zero and non-zero cases are presented in Appendix~\ref{lemma_eigen}.

Next, we discuss formalization of the boundedness of the matrix norm of $E_{h}=A_{h}^{-1}$. We have used an explicit formulation of $A_{h}^{-1}$~\cite{hu1996analytical} in our formalization and we verify this formally using the definition: $A_{h}^{-1}A_{h}=I \; \land \;  A_{h}A_{h}^{-1}=I$. Details on the proof can be referred to in the Appendix~\ref{invertible_check}.

\subsection{ Lemma on the boundedness of the matrix norm for scheme $\mathcal{N}_{h}$}
Here, we have used the definition of the spectral (2-norm): $||A||_{2} = \rho(A)$,
where $\rho(A)$ is the spectral radius of $A$ and is defined as the maximum eigen-value of A, i.e. $\rho(A)=max_{m} |\lambda_{m}(A)|$. 
For the symmetric tri-diagonal matrix $A_{h}$, $A=E_{h}$ and $\lambda_{m}(E_{h})= 1/ \lambda_{m}(A_{h})$. 
Since $\lambda_{m} (A_{h}) < 0$, $max_{m} |\lambda_{m}(E_{h})|= 1/ |\lambda_{min}(A_{h})|$. Hence, we define the matrix norm in Coq as follows:
\begin{small}
\begin{verbatim}
Definition matrix_norm (N:nat):= 1/ Rabs (Lambda_min N).
\end{verbatim}
\end{small} To show that the matrix norm is uniformly bounded, we need to show that $1/ |\lambda_{min}(A_{h})|$ is uniformly bounded. This is where we instantiate the tri-diagonal matrix $A_{h}$ with the scheme $\mathcal{N}_{h}$. Thus, we prove the following lemma in Coq:
\begin{small}
\begin{verbatim}
Lemma spectral: forall(N:nat),(2<N)%nat -> 1/Rabs(Lambda_min N) <= L^2/4.
\end{verbatim}
\end{small}where $L$ is the length of the domain, independent of $h$, and is constant throughout. \texttt{Lambda\_min} is the minimum eigenvalue for the instantiated matrix,
$A_{h}'= A_{h}(\frac{1}{h^2}, \frac{-2}{h^2}, \frac{1}{h^2})$. We provide a paper proof of this bound in the Appendix~\ref{paper_proof}.

To show that all the eigenvalues have the same bound, we prove that $\frac{1}{\lambda_{min}(A_{h}')}$ is the maximum eigenvalue of $E_{h}'$. The lemma statement is as follows:
\begin{small}
\begin{verbatim}
Lemma eigen_relation: forall (i N:nat), (2<N)%nat ->(0<=i<N)%nat -> 
    Rabs (lam i N) <= 1/ Rabs( Lambda_min N).
\end{verbatim}
\end{small}This completes the proof on the boundedness of the eigenvalues of $E_{h}'$. The lemma, \texttt{eigen\_relation} also shows that the spectral radius of $E_{h}'$ is $\frac{1}{|\lambda_{min}(A_{h}')|}$, and justifies the defintion of \texttt{matrix\_norm}.

We note that the definition of the matrix norm of $A_{h}^{-1}$ is valid only if $A_{h}^{-1}$ is a normal matrix . We therefore verify that  $A_{h}^{-1}$ is normal. The lemma statement is provided in the Appendix~\ref{inverse_normal}. 

We also provide the proof that $A_{h}$ is diagonalizable in the Appendix~\ref{diagonalization}. This helps us to formally establish that the eigen vectors are orthogonal and hence the eigen space is complete.
\subsection{Main stability theorem}
In this section, we integrate all of the previous lemmas to prove the main stability theorem (\ref{stability}).
\begin{small}
\begin{verbatim}
Theorem stability: forall (u:X) (f:Y) (h:R) (uh: Xh h)
    (rh: forall (h:R), X -> (Xh h))(sh: forall (h:R), Y->(Yh h)) 
    (E: Y->X) (Eh:forall (h:R), (Yh h)->(Xh h)),
    exists K:R , forall (h:R), operator_norm(Eh h)<=K.
\end{verbatim}
\end{small}where the operator norm is instantiated with the matrix norm using the following hypothesis:
\begin{small}
\begin{verbatim}
Hypothesis mat_op_norm: forall (u:X) (f:Y) (h:R) (uh: Xh h)
    (rh: forall (h:R), X -> (Xh h))(sh: forall (h:R), Y->(Yh h)) 
    (E: Y->X) (Eh:forall (h:R),(Yh h)->(Xh h)),
    operator_norm (Eh h) = matrix_norm m.
\end{verbatim}
\end{small}

\section{Application of the Lax equivalence theorem to the example problem}
\label{Lax_apply}
In this section, we apply the Lax equivalence theorem that we proved in Section~\ref{Lax_section} to a concrete differential equation $\frac {d^{2}u}{dx^{2}}=1$ and the numerical scheme $\mathcal{N}_h$ given by $\frac { u_{i+1}-2u_{i}+u_{i-1}}{\Delta x^2} = 1$. We recall that the proof of convergence using the Lax equivalence theorem requires that the difference scheme is consistent with respect to the differential equation and is stable. We discussed the proof of consistency of the scheme in Section~\ref{finite} and the stability in Section~\ref{stability_section}. Thus, we apply these proofs to complete the proof of convergence for the scheme. We provide the theorem statement to verify convergence of the scheme in the Appendix~\ref{appendix_A}.

\section{Conclusion and Future work}\label{conclusion}
This work investigated the formalization of convergence, stability and consistency of a finite difference scheme in the Coq proof assistant. Any continuously differentiable function can be approximated by a Taylor polynomial.  The Lagrange remainder of a Taylor series provides an estimate of the \textcolor{black}{truncation} error and we formally proved that this error can be bound by $n^{th}$ power of the discretization step, $\Delta x$, where $n-1$ is the order of the Taylor
polynomial. We  implemented the proof of the consistency of a finite difference scheme by breaking down the theorem statement into lemmas, each corresponding to function values at points neighboring the point of evaluation. These lemmas were  proved individually by applying the Taylor--Lagrange theorem, the proof of which is already formalized in the \texttt{Coq.Interval} library \cite{martin2013certified}.  
Consistency and stability guarantees convergence as stated by the Lax equivalence theorem. Following the proof of the the Lax equivalence theorem, we formally proved convergence of a specific finite difference scheme. Specifically, we proved that the global discretization error could be bounded above by a constant times the local discretization error.
Then, by applying the sandwich theorem for limits, we proved that the convergence condition is satisfied in the limit $\Delta x \to 0$. In the process of formalizing the proof of stability for the numerical scheme, we also developed tools for linear algebra and spectral theory, for the \texttt{Coquelicot} definition of matrices in Coq, which can be reused. As noted earlier, the approach we follow is not specific to the sample numerical scheme, but can be easily extended to other numerical schemes with appropriate \textcolor{black}{instantiation} of the matrix $A_{h}$, and vectors, $r_{h}u$, $s_{h}Au$. Formalization of the proof of orthogonality of the eigenvectors helped us report the missing constant $\sqrt{\frac{2}{N+1}}$ in $s_{m}$ that occurs in most textbooks/literature on numerical analysis.

This work  considered the impact of the discretization error on the convergence of a numerical method to the exact solution. In a practical setting, floating point errors have to be also accounted for, as an accumulation of such errors can lead to deviations from the true solution. In future work, we will extend our results to incorporate floating point errors and their impact on the convergence of finite difference numerical schemes. We also plan on working with iterative solvers, which would be an extension of our current work on direct solvers (explicit inversion of the matrix $A_{h}$). We also plan on working with the Frama-C toolkit~\cite{10.1007/978-3-642-33826-7_16} for verification \textcolor{black}{of existing programs} and be able to discharge the generated verification conditions using the Coq proofs we present in this paper.
%\jb{talk about your work with implementations and Frama-C}

\subsection{ Effort and challenges: }
The total length of the Coq code and proofs is about 14,000 lines, \textcolor{black}{of which about 1200 lines are specific to the scheme. The rest of the formalization can be reused for a generic symmetric tridiagonal matrix}. \textcolor{black}{It took us about 15 months for the entire formalization.} Much of the effort was spent on destructing the matrices and developing required linear algebra tools to handle the matrix manipulation. Since we are treating stability from a spectral point of view, lack of spectral theory for numerical analysis for the \texttt{Coquelicot} definition of matrices has been challenging for us.
 For the proof of consistency, the primary challenge was the right placement of the quantifiers to bound the Lagrange remainder using the definition of big-$O$ notation. %We also had some issues with existential quantification. 
To instantiate $\Gamma =M+K$, we had to carefully destruct the lemmas into the main theorem. \textcolor{black}{We believe that a generic library with an automated implementation of the big-O definitions would save considerable effort here.} We also encountered issues in selecting appropriate instantiations for other existential parameters. In the proof of convergence, we had to carefully construct the application of properties of limit with filters of neighborhoods. 
\newpage

\bibliographystyle{splncs04}
\bibliography{sample-base}
\appendix
\section{Proof of the uniform bound on the eigen values of $A_{h}(1/h^2, -2/h^2,1/h^2)$}\label{paper_proof}
In this section, we provide a paper proof of the uniform boundedness of the eigenvalues of the scheme $\mathcal{N}_h$.
\begin{proof}\small
\begin{align*}
&\lambda _{min}(A_{h}') = \frac{2}{h^2} \left[-1+\cos \left(\frac{\pi}{N+1}\right)\right] \quad \text{[For m=1 in the expression of $\lambda_m$]}\\
&\text{Since all eigenvalues are negative, $\text{min} |\lambda_{m} (A_{h}')|= | \lambda_{min}(A_{h}')|$}, \\
\therefore& \frac{1}{|\lambda_{min}(A_{h'})|}= \frac{1}{\left| \frac{2}{h^2}\left[-1+ \cos \left(\frac{\pi}{N+1}\right)\right]\right|} \implies \frac{1}{|\lambda_{min}(A_{h}')|}= \frac{h^2}{4\sin^{2}\left(\frac{\pi}{2 (N+1)}\right)} \\
&\text{[Using the identity: $-1+\cos(2x)= -2 \sin^{2}(x)$]}\\
& \text{Using the definition, $h\overset{\Delta}{=} \frac{L}{N+1}$, where L is the domain length,}\\
& \therefore \frac{1}{|\lambda_{min}(A_{h}')|}=\frac{L^{2}}{4 (N+1)^{2}\sin^{2}\left(\frac{\pi}{2(N+1)}\right)}= \frac {L^2}{\pi^2}\frac{\pi^2}{4 (N+1)^2 \sin^{2} \left(\frac {\pi}{2(N+1)}\right)}=\frac{L^2}{\pi^2}\frac{x^2}{\sin^{2}(x)}\\
&\text{where}, x=\frac{\pi}{2(N+1)}\\
&\text{Using the relation, } \forall x \in (0,\pi/2],\quad \frac{2x}{\pi}\leq \sin(x), \text{ or, } \frac{x}{\sin(x)} \leq \frac{\pi}{2}, \text{we get :}\frac{x^2}{\sin^{2}(x)}\leq \frac{\pi^2}{4}\\
& \therefore \frac{1}{|\lambda_{min}(A_{h}')|}\leq \frac{L^2}{4} \quad \textbf{[QED]}
\end{align*}
\end{proof} We prove the relation $\forall x \in (0, \pi/2], \quad \frac{x}{\sin(x)} \leq \frac{\pi}{2}$, by using the concavity of $\sin(x)$ in $[0, \pi/2]$.
We define a concave function $f: \mathbb{R}\to \mathbb{R}$ in Coq as follows:
\begin{small}
\begin{verbatim}
Definition concave (f:R->R) (x y c:R):=
    0<=c<=1 -> f(c*x + (1-c) * y) >= c* f x + (1-c) * f y.
\end{verbatim}
\end{small}The proof for $\frac{x^2}{\sin^{2}(x)} \leq \frac{\pi^2}{4}, \; \forall x \in (0, \pi/2]$ is formalized as the following lemma statement in Coq:
\begin{small}
\begin{verbatim}
Lemma spectral_intermed:forall(x:R),0<x<=PI/2 ->(x^2)/(sin x)^2 <=(PI^2)/4. 
\end{verbatim}
\end{small}

\section{Lemmas required to complete the proof of stability:}
\subsection { Lemma to verify the invertibility of $A_{h}$}\label{invertible_check}
In this subsection, we verify that the explicit form of the inverse~\cite{hu1996analytical} we use is indeed the inverse of $A_{h}$, i.e. $A_{h}A_{h}^{-1}= A_{h}^{-1} A_{h}=I$. 
In Coq, we state the following lemma to verify the invertibility of $A_{h}$:
\begin{small}
\begin{verbatim}
Lemma invertible_check (a b:R) : forall (N:nat), (2<N)%nat -> 0<a -> 
    Mk N (b/a) <> 0 -> invertible N (Ah N a b a ) (inverse_A N a b ).
\end{verbatim}
\end{small} Here, $M_{k}$ is the determinant of $A_{h}$ of size $k$. We used the recurrence relation~\cite{hu1996analytical}:  $M_{k} = D \times M_{k-1} - M_{k-2}, \quad D=\frac{b}{a}$. Overall, the approach is similar to the proof of the lemma \texttt{eigen\_belongs}, i.e. we exploit the tridiagonal structure of $A_{h}$.
The proof required us to formalize some properties from combinatorics.

For the scheme that we are considering, $D=-2$. Two important steps that were required to complete the proof of $M_{k} \neq 0$ for the scheme $\mathcal{N}_{h}$ were:
\begin{enumerate}
    \item Proving that $M_{k}=(-1)^k \times (k+1)$: We proved this using strong induction on $k$ and the recurrence relation described above. To get an intuition of why it is true, we observe the values of $M_{k}$ for initial values of $k$: $M_{0} = 1, \quad M_{1}=-2 ,\quad M_{2}=3, \quad M_{3}=-4 \cdots M_{k}= (-1)^k \times (k+1)$
    \item Proving that the determinant, $M_{k} \neq 0$
\end{enumerate}

\subsection{Lemmas on spectrum of $E_{h}$}\label{inverse_spectrum}
In this subsection, we prove that the eigenvalues of $E_{h}$ are just inverse of the eigenvalues of $A_{h}$, while the eigenvectors are the same. This follows from the following informal proof:
\begin{small}
\begin{proof} We start with the definition of Eigen-system $(\lambda_m, s_m)$, 
\begin{align}
&A_{h}s_{m}=\lambda_{m}s_{m} \nonumber\\
&\text{Multiplying by $A_{h}^{-1}$ on both sides and using the definition :} A_{h}^{-1}A_{h}=I, \nonumber\\
&A_{h}^{-1}A_{h}s_{m}=A_{h}^{-1}\lambda_{m}s_{m} \implies s_{m}=\lambda_{m}A_{h}^{-1}s_{m} \nonumber \implies \frac{s_{m}}{\lambda_{m}}=A_{h}^{-1}s_{m} \nonumber
\end{align}
\end{proof}
\end{small}In Coq, we define the following lemma to formalize this proof:
\begin{small}
\begin{verbatim}
Lemma inverse_eigen (m N:nat) (a b:R) : (2< N)%nat -> (0<=m<N)%nat ->
 0<a -> ((invertible N (Ah N a b a) (inverse_A N a b)) /\ 
 (LHS m N a b a= RHS m N a b a)) ->(Eigen_vec m N a b a) =
 Mmult (inverse_A N a b) (Mmult (Eigen_vec m N a b a) (Lambda m N a b a)).
\end{verbatim}
\end{small}

\subsection{Lemma to verify that the $A_{h}^{-1}$ is normal}\label{inverse_normal}
In this subsection, we verify that $A_{h}^{-1}$ is normal. This lemma is stated as:
\begin{verbatim}
Lemma inverse_is_normal (a b:R): forall (N:nat),
 Mmult (inverse_A N a b ) (mat_transpose N (inverse_A N a b )) = 
 Mmult (mat_transpose N (inverse_A N a b )) (inverse_A N a b ).
\end{verbatim}

\subsection{Intermediate lemmas to complete the proof of the lemma eigen\_belongs:}\label{lemma_eigen}
This proof requires some intermediate lemmas which verify certain properties which are as follows:
\subsubsection{Lemmas on structure of the matrix:} 
In this subsection, we provide the lemmas that verify the structure of the matrix, i.e. the diagonal entries are $b$, the sub-diagonal entries are $a$ and the super-diagonal entries are $c$.
Mathematically, the lemmas say: $A_{h}(i,i)=b$, $A_{h}(i-1,i)=a$, $A_{h}(i,i+1)=c\quad \forall i=1\cdots N-2$. For the first and last rows, we have the structure as: $A_{h}(0,0)=b, A_{h}(0,1)=c, A_{h}(N-1,N-2)=a$ and $A_{h}(N-1,N-1)=b$. In Coq, we define the following lemmas to verify the above-mentioned structure:
\begin{small}
\begin{verbatim}
Lemma coeff_prop_1 (a b c:R): forall (i N:nat), (2<N)%nat ->
  (0<i <N)%nat -> coeff_mat Hierarchy.zero (Ah N a b c) i (pred i) = a .

Lemma coeff_prop_2 (a b c:R): forall (i N:nat), (2<N)%nat ->
    (i <N)%nat ->coeff_mat Hierarchy.zero (Ah N a b c) i i = b.

Lemma coeff_prop_3 (a b c:R): forall (i N:nat), (2<N)%nat -> 
  (i< pred N)%nat -> coeff_mat Hierarchy.zero (Ah N a b c) i (i + 1) = c.
\end{verbatim}
\end{small}

\subsubsection{Lemmas to handle the zeros case: }
\begin{figure}
\includegraphics[width=\textwidth]{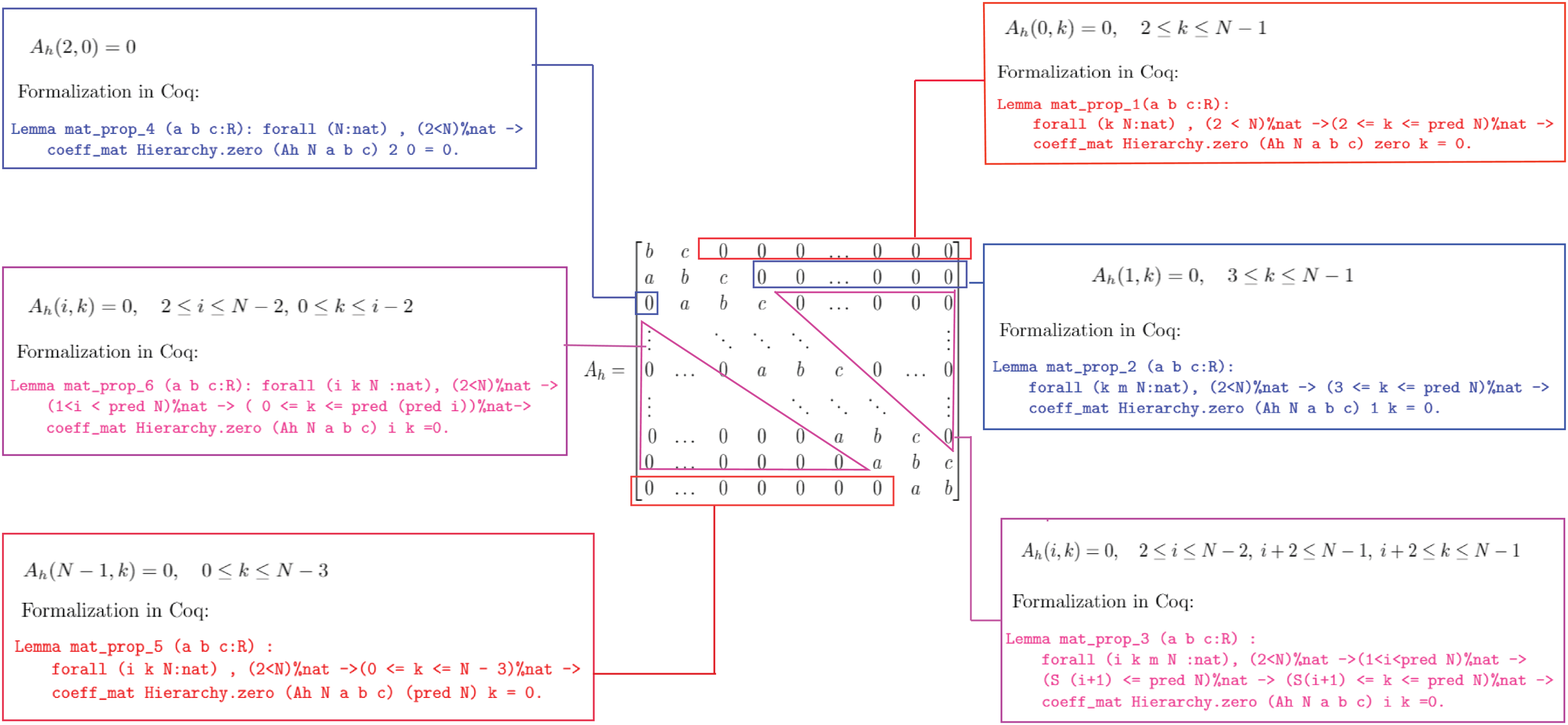}
\caption{Formalizing the tri-diagonal structure of the matrix. This formalization can be used for any tri-diagional system. } \label{tridiagonal}
\end{figure}
A good amount of effort was also required in extracting zero entries in the matrix $A_{h}$ and proving that their sum equals zero. This again exploits the structure of the matrix, illustrated in figure~(\ref{tridiagonal}). Two important lemmas that played a pivotal role in this proof are :
\begin{small}
\begin{verbatim}
Lemma sum_const_zero:forall(n m:nat),(n<=m)%nat-> sum_n_m(fun _=>0)n m=0.
\end{verbatim}
\end{small}
Mathematically this means:  $\forall n,m:nat, (n \leq m), \sum_{n}^{m} 0=0$
\begin{small}
\begin{verbatim}
Lemma sum_n_m_zero(a:nat -> G)(n m:nat):(m<n)%nat -> sum_n_m a n m = zero.
\end{verbatim}
\end{small}
Mathematically this means: $ \forall (a: nat \rightarrow G), (n,m:nat)$, $ (m<n), \; \sum_{n}^{m}a =0$,
where, $a$ is a function from naturals to an abelian group (G), in our case, it is reals.
The first lemma was proved by us but the second one is already present in the \texttt{Coquelicot} library.
\subsubsection{Lemmas to handle the non-zero case:}
The other part of the proof is to equate the sum of non-zero entries in LHS to a non-zero entry in RHS. i.e. $\vec{Ah}_{i} \cdot \vec{s_{m}}=\lambda_{m} s_{mi}$, where the $Ah_{i}$ represents the $i^{th}$ row of the matrix and $s_{mi}$ denotes the $i^{th}$ component of the Eigen-vector $s_{m}$ and $\lambda_{m}$ is a scalar. In Coq, considering $i=0$, for example, would translate to the lemma statement:
\begin{small}
\begin{verbatim}
Lemma i_0_j (a b c:R):
forall (m N:nat), (2<N)%nat -> (0<=m<N)%nat -> a=c /\ 0<c->
  mult (coeff_mat Hierarchy.zero (Ah N a b c ) zero 0)
    (coeff_mat Hierarchy.zero (Eigen_vec m N a b c ) 0 0) +
  mult (coeff_mat Hierarchy.zero (Ah N a b c ) zero 1)
    (coeff_mat Hierarchy.zero (Eigen_vec m N a b c ) 1 0) =
  mult (coeff_mat Hierarchy.zero (Eigen_vec m N a b c ) zero 0)
    (coeff_mat Hierarchy.zero (Lambda m N a b c ) 0 0).
\end{verbatim}
\end{small}
We are not providing other lemmas here in the interest of space, but they can be referred in the attached code.

\section{ Diagonalization of $A_{h}$}\label{diagonalization}\label{diagonalization}
In this section, we discuss the lemmas required to prove that $A_{h}$ is diagonalizable, i.e. $A_{h}= S \Lambda S^{T}$, where S is the matrix of eigenvectors and $\Lambda$ is a diagonal matrix of Eigen-values of $A_{h}$. We first present an informal proof:
\begin{proof}We start with the definition of an Eigensystem:
\begin{equation*}
A_{h}S=S \Lambda \implies A_{h}SS^{T}=S\Lambda S^{T} \implies A_{h}=S \Lambda S^{T}\quad [SS^{T}=I] 
\end{equation*}
\end{proof}Here, we use the fact that $S^{-1}=S^{T}$, since S is orthonormal. We verify this by using the definition of inverse of matrices, i.e. $SS^{T}=S^{T}S=I$. In Coq, we prove the following lemma:
\begin{small}
\begin{verbatim}
Lemma Scond:forall (N:nat) (a b:R), (2<N)%nat -> 0<a ->  
    Mmult (Sm N a b) (Stranspose N a b) = identity N /\ 
        Mmult (Stranspose N a b) (Sm N a b) = identity N.
\end{verbatim}
\end{small}To prove the lemma \texttt{Scond}, we split the proof into two sub-proofs:
\begin{enumerate}
\item i = j, 
\item $i \neq j$
\end{enumerate}For the first case, we have the condition that $\vec{s}_{i} \cdot \vec{s}_{i}=1$, i.e. $||\vec{s}_{i}||^{2}=1$. This reduces to proving that the sum of the following sine-squared series is 1.
\begin{small}
\begin{equation}\label{sin_sqr_sum}
\sum_{m=1}^{N}\frac{2}{N+1}\sin^{2}\left[j \frac{m\pi}{N+1}\right]=1
\end{equation}
\end{small}In Coq, we prove the following lemma to verify (\ref{sin_sqr_sum}):
\begin{small}
\begin{verbatim}
Lemma sin_sqr_sum: forall (i N:nat), (2<N)%nat /\ (0<=i<N)%nat -> 
    sum_n_m (fun l:nat => (2/(INR(N+1)))*
        sin(((INR l+1)* INR (i+1)*PI)*/ INR (N+1)) ^2) 0 (pred N)=1.
\end{verbatim}
\end{small} Here, we make use of the following theorem from~\cite{knapp2009sines}:
\begin{theorem}\label{sin_sqr}
    If $a\;b \in \mathbb{R}$ and $d \neq 0$ and n is a positive integer, \\$\sum_{k=0}^{n-1}cos(a+kd)=\frac{\sin{nd/2}}{\sin{d/2}}cos\left(a+\frac{(n-1)d}{2}\right)$
\end{theorem}
where $\sin^{2}{(\theta)} = (1-\cos {(2\theta)})/2$. We state the Theorem~\ref{sin_sqr}, using the following hypothesis statement in Coq:
\begin{small}
\begin{verbatim}
Hypothesis cos_series_sum: forall (a d:R) (N:nat), d <>0->
    sum_n_m (fun l:nat => cos (a+(INR l)*d)) 0 (pred N)= 
    sin(INR N*d/2)* cos(a+INR(N-1)*d/2)*/ sin(d/2).
\end{verbatim}
\end{small}We then use the hypothesis \texttt{cos\_series\_sum} to prove the lemma \texttt{sin\_sqr\_sum}.\\
\mohit{details added here}
For the second case, we have the orthogonality condition $\vec{s}_{i} \cdot \vec{s}_{j}=0, \; i \neq j$. This reduces to proving:
\begin{equation}\label{neq_1}
\sum_{k=0}^{N-1} \sin{\left[ (k+1) \frac{(i+1)\pi}{N+1}\right]} \sin{\left[(k+1) \frac{(j+1)\pi}{N+1}\right]}=0
\end{equation} since, $\frac{2}{N+1}$ is a constant, it can be taken outside the summation.\\
Using the trigonometric identity, $$\sin{A}\sin{B}=\frac{1}{2} \left[\cos{(A-B)} - \cos{(A+B)} \right]$$ we can reduce~(\ref{neq_1}) into sums of cosines as follows:
\begin{equation}\label{neq_2}
    \frac{1}{2}\sum_{k=0}^{N-1} \cos{\Big[(k+1)\frac{(i-j)\pi}{N+1}\Big]}-
    \frac{1}{2}\sum_{k=0}^{N-1} \cos{\Big[(k+1)\frac{(i+j+2)\pi}{N+1}  \Big]}=0
\end{equation}
Using Theorem~(\ref{sin_sqr}), we can further reduce each sum in equation~(\ref{neq_2}) into the product of sine and cosine. By doing some algebra, we prove that if $(i-j)$ and $(i+j+2)$ are simultaneously even or they are simultaneously odd, the sums in equation~(\ref{neq_2}) cancel out. We further note that it is always the case that $(i-j)$ and $(i+j+2)$ are simultaneously even or they are simultaneously odd. We provide an informal proof of this fact as follows:
\begin{proof}
Case 1: $(i-j)$ is even: 
\begin{align}
    &\exists m:nat, (i-j) = 2m \nonumber \\
    & \implies i= 2m+j \nonumber \\
    & \implies i+j+2= 2m+j+j+2 \nonumber \\
    & \implies i+j+2 = 2* (m+j+1) \quad \therefore  \text{ Even} \nonumber
\end{align}
Case 2: $(i-j)$ is odd:
\begin{align}
    & \exists m:nat, (i-j) = 2m+1 \nonumber \\
    & \implies i= j+ 2m +1 \nonumber \\
    & \implies i+j+2= j+2m+1+j+2 \nonumber \\
    & \implies i+j+2= 2* (j+m+1)+1 \quad \therefore \text{ Odd} \nonumber
\end{align}
\end{proof}and vice-versa for each cases. This completes the proof of orthogonality of the Eigen vectors. In Coq, we prove the following lemma to verify~(\ref{neq_2}):
\begin{small}
\begin{verbatim}
Lemma cos_sqr_sum: forall (i j N:nat), 
(2<N)%nat /\ (0<=i<N)%nat /\ (0<=j<N)%nat /\ (i<>j) -> 
sum_n_m (fun l:nat => mult(/INR(N+1))
    (cos((INR(i) - INR(j)) * PI / INR (N + 1) + 
        INR l * (INR(i) - INR(j)) * PI / INR (N + 1)) - 
    cos(INR(i+j+2)*PI */ INR(N+1) + 
        INR l * INR(i+j+2)*PI */ INR(N+1)))) 0 (pred N)=0.
\end{verbatim}
\end{small}

\section{Application of the Lax--Equivalence theorem to the scheme $\mathcal{N}_h$}\label{appendix_A}
We define the following theorem statement to prove the convergence of the numerical scheme in Coq:
\begin{small}
\begin{verbatim}
Theorem scheme_convergence: forall (U:X) (f:Y) (h:R) (uh: Xh h) 
 (rh: forall (h:R), X -> (Xh h)) (sh: forall (h:R), Y->(Yh h)) (E: Y->X) 
 (Eh:forall (h:R), (Yh h)->(Xh h)), 
 is_linear_mapping X Y Aop ->  f=Aop U->
 (* Hypothesis that A is a linear mapping from X to Y*)
 (forall (h:R),is_linear_mapping (Xh h) (Yh h) (Ah_op h))-> 
 (* Hypothesis that Ah is a linear. mapping from Xh to Yh for each h*)
  (forall (h:R), is_bounded_linear X (Xh h) (rh h))->
  (* Hypothesis that rh is a bounded linear
  operatior (restriction) from X to Xh for each h*)  
  (forall (h:R),is_bounded_linear Y (Yh h) (sh h))-> 
  (* Hypothesis that sh is a bounded linear
  operator (restriction) from Y to Yh*) 
  is_bounded_linear Y X E -> 
  (* Hypothesis that E is a bounded linear operator from Y to X*)
  U=E f-> 
  (* Defining solution in continuous space (true solution)*)
  (forall (h:R), is_bounded_linear (Yh h) (Xh h) (Eh h)) -> 
  (* Hypothesis that Eh is a bounded linear operator from Yh to Xh for each h*)
  (forall h:R, is_finite (operator_norm(Eh h))) -> 
   (* Hypothesis that ||Eh|| is finite*)
  (uh= Eh h (sh h f))-> 
   (* Defining a discrete solution uh*)
  ( Ah_op h uh = sh h f)-> (*f =fh*)
  (forall (h:R), rh h U= Eh h (Ah_op h (rh h U)))->
  (*uh =Eh *Ah *uh, where Eh*Ah=I*)
  (forall h:R, minus (Ah_op h (rh h U)) (sh h (Aop U)) <> Hierarchy.zero)->
  
  is_lim (fun h:R= norm (minus (rh h (E(f))) (Eh h (sh h (f))))) 0 0 .
    (*Convergence*)
\end{verbatim}
\end{small}

\end{document}